\documentclass[11pt]{amsart}
\title{Misiurewicz points for polynomial maps and transversality}

\usepackage{color}
\usepackage{amssymb,amsmath,amsthm,fullpage}
\usepackage[all]{xy}  
\usepackage{url}
\usepackage{rotating} 
\usepackage{tikz}
\usepackage{textcomp,listings}
\usepackage{versions}

\definecolor{green}{rgb}{0,0.5,0}
\definecolor{dkgreen}{rgb}{0,0.6,0}
\definecolor{gray}{rgb}{0.5,0.5,0.5}
\definecolor{mauve}{rgb}{0.58,0,0.82}
\lstnewenvironment{python}[1][]{
\lstset{
  frame=single,                   
  language=python,                          
  basicstyle=\ttfamily\small, 
  numbers=left,                             
  numberstyle=\scriptsize\color{black},  
  stepnumber=1,                   
  numbersep=7pt,                  
  backgroundcolor=\color{white},      
  stringstyle=\color{red},
  showspaces=false,               
  showstringspaces=false,         
  showtabs=false,                 
  tabsize=2,                      
  captionpos=b,                   
  breaklines=true,                
  breakatwhitespace=false,        
  rulecolor=\color{black},        
  framexleftmargin=1mm, framextopmargin=1mm, frame=shadowbox, 
  alsoletter={1234567890},
  otherkeywords={\ , \}, \{},
  keywordstyle=\color{blue},       
  stringstyle=\color{mauve},         
  emph={access,and,break,class,continue,def,del,elif ,else,%
  except,exec,finally,for,from,global,if,import,in,i s,%
  lambda,not,or,pass,print,raise,return,try,while},
  emphstyle=\color{black}\bfseries,
  emph={[2]True, False, None, self},
  emphstyle=[2]\color{blue},
  emph={[3]from, import, as},
  emphstyle=[3]\color{blue},
  upquote=true,
  morecomment=[s]{"""}{"""},
  commentstyle=\color{dkgreen},       
  emph={[4]1, 2, 3, 4, 5, 6, 7, 8, 9, 0},
  emphstyle=[4]\color{blue},
  literate=*{:}{{\textcolor{blue}:}}{1}%
  {=}{{\textcolor{blue}=}}{1}%
  {-}{{\textcolor{blue}-}}{1}%
  {+}{{\textcolor{blue}+}}{1}%
  {*}{{\textcolor{blue}*}}{1}%
  {!}{{\textcolor{blue}!}}{1}%
  {(}{{\textcolor{blue}(}}{1}%
  {)}{{\textcolor{blue})}}{1}%
  {[}{{\textcolor{blue}[}}{1}%
  {]}{{\textcolor{blue}]}}{1}%
  {<}{{\textcolor{blue}<}}{1}%
  {>}{{\textcolor{blue}>}}{1},%
}}{}

\definecolor{orange}{rgb}{1,0.65,0.17}

\def\C{\mathbb{C}}

\def\Z{\mathbb{Z}}
\def\Q{\mathbb{Q}}
\def\P{\mathbb{P}}
\def\A{\mathbb{A}}

\def\F{\mathbb{F}}
\def\N{\mathbb{N}}
\def\O{\mathcal{O}}

\newcommand{\col}{\,{:}\,}
\newcommand{\sst}{^{\operatorname{st}}}
\newcommand{\tth}{^{\operatorname{th}}}

\DeclareMathOperator{\Res}{Res}
\DeclareMathOperator{\Disc}{Disc}

\theoremstyle{plain}
\newtheorem{thm}{Theorem}[section]
\newtheorem*{thm*}{Theorem}
\newtheorem{lem}[thm]{Lemma}
\newtheorem{prop}[thm]{Proposition}
\newtheorem{cor}[thm]{Corollary}
\theoremstyle{definition}
\newtheorem*{defn}{Definition}

\newtheorem{exmp}{Example}[section]
\newtheorem*{exmp*}{Example}

\theoremstyle{remark}
\newtheorem*{rem}{Remark}

\author{Benjamin Hutz}
\address{
Department of Mathematics \\
Florida Institute of Technology \\
Melbourne, FL}
\email{bhutz@fit.edu}

\author{Adam Towsley}
\address{
Department of Mathematics\\
Elmira College\\
Elmira, NY
}
\email{atowsley@elmira.edu}

\excludeversion{code}

\begin{document}
\begin{abstract}
    The behavior under iteration of the critical points of a polynomial map plays an essential role in understanding its dynamics.
    We study the special case where the forward orbits of the critical points are finite.
    Thurston's theorem tells us that fixing a particular critical point portrait and degree leads to only finitely many possible polynomials (up to equivalence) and that, in many cases, their defining equations intersect transversely.
    We provide
    explicit algebraic formulae for the parameters where the critical points of the unicritical polynomials and bicritical cubic polynomials have a specified exact period. We pay particular attention to the parameters where the critical orbits are strictly preperiodic, called Misiurewicz points. Our main tool is the generalized dynatomic polynomial.  We also study the discriminants of these polynomials to examine the failure of transversality in characteristic $p>0$ for the unicritical polynomials $z^d + c$.
\end{abstract}

\maketitle
\section{Introduction and Results}
    The behavior under iteration of the critical points of a polynomial $f$ plays an essential role in understanding the dynamics of $f$. We study the special case where the forward orbits of the critical points are finite; such maps are called \emph{post-critically finite} (PCF). The goal of this paper is to construct (algebraically) all the PCF polynomials of any degree with one (affine) critical point (unicritical) or PCF polynomials of degree 3 with two (affine) critical points (bicritical). Additionally, we study the failure of transversality for the unicritical family in characteristic $p>0$.

    The unicritical polynomials are elements of the families
    \begin{equation*}
        f_{d,c}(z) = z^d +c
    \end{equation*}
    with $d \geq 2$ and $c \in \C$ with critical point $0$. The cubic bicritical polynomials are elements of the family
    \begin{equation*}
        g_{a,v}(z) = z^3 - 3a^2z + (2a^3 + v)
    \end{equation*}
    with $a,v \in \C$ and critical points $\pm a$ and critical value $g_{a,v}(a) = v$.

    Given a polynomial $f$ we denote its $n$-th iterate by $f^n ( z ) = f \circ f^{n-1}(z)$.
    We say $z_0$ is a \emph{preperiodic point with period} $(m,n)$ for $f$ if $f^{m+n}(z_0) = f^m(z_0)$. We call $m$ the \emph{preperiod}. For a preperiodic point of period $(m,n)$, we say that it has \emph{exact period $(m,n)$} if it is does not have period $(k,t)$ for any $k \leq m$ and $t \mid n$ with at least one of $k<m$ or $t < n$.
    If the \emph{preperiod} $m$ is positive, we say that the point is \emph{strictly preperiodic}.
    For example a strictly preperiodic point $z_0$ with exact period $(3,2)$ has orbit
    \begin{equation*}
            \xygraph{
            !{<0cm,0cm>;<1cm,0cm>:<0cm,1cm>::}
            !{(0,0) }*+{z_0}="a"
            !{(1,0) }*+{\bullet}="b"
            !{(2,0) }*+{\bullet}="c"
            !{(3,0) }*+{\bullet}="d"
            !{(4,0) }*+{\bullet}="e"
            "a":"b"
            "b":"c"
            "c":"d"
            "d":@(rd,ru)"e"
            "e":@(rd,ru)"d"
            }
    \end{equation*}
    \begin{defn}
        In the unicritical case we say that $c_0$ is a \emph{Misiurewicz point} of period $(m,n)$ if the orbit of the critical point $0$ by $f_{d,c_0}$ has exact period $(m,n)$ and $m >0$.
        In the cubic bicritical case, we define a \emph{Misiurewicz point} as a pair of parameter values $(a,v)$ such that both critical points, $\pm a$, are strictly preperiodic for $g_{a,v}$.
    \end{defn}

    To analyze the Misiurewicz points algebraically, both for unicritical and bicritical cubic polynomials, we use the generalized dynatomic polynomial, \cite{Hutz},
    \begin{equation*}
        \Phi^{\ast}_{f,m,n}(z) =  \frac{\Phi^{\ast}_{f,n}(f^m(z))}{\Phi^{\ast}_{f,n}(f^{m-1}(z))},
    \end{equation*}
    where
    \begin{equation*}
        \Phi^{\ast}_{f,n}(z) = \prod_{k \mid n} f^k(z)^{\mu(n/k)}
    \end{equation*}
    is the dynatomic polynomial defined in \cite{SilvermanADS}. The function $\mu$ is the M\"obius function which satisfies $\mu(n) =0$ for $n$ not squarefree, $\mu(1)=1$, and $\mu(n) = (-1)^k$ when $n$ has $k$ distinct prime factors.
    While the points of exact period $(m,n)$ are among the roots of the polynomial $\Phi^{\ast}_{f,m,n}$, for $m \neq 0$, even in the case the roots are simple, it is not necessarily the case that the roots are precisely the points of exact period $(m,n)$.
    \begin{exmp}
        Let $f(z) = z^2-1$. Then we compute
        \begin{equation*}
            \Phi^{\ast}_{f,1,2}(z) = z(z-1).
        \end{equation*}
        However, $0$ is a preperiodic point with minimal period $(0,2)$.
    \end{exmp}
    This makes the construction of a polynomial whose roots are precisely the Misiurewicz points of exact period $(m,n)$ more subtle. In Section \ref{sect_uni}, we prove an explicit algebraic construction of all Misiurewicz points for the unicritical family.
    \begin{thm}\label{thm1}
        The $c$ values for which $f_{d,c}(z) = z^d + c$ is post-critically finite with $0$ having exact period $(m,n)$ are the roots of
        \begin{equation*}
            G_d(m,n)=
            \begin{cases}
              \frac{\Phi_{f,m,n}^{\ast}(0)}{\Phi_{f,0,n}^{\ast}(0)^{d-1}} & \text{if } m \neq 0 \text{ and } n \mid (m-1)\\
              \Phi_{f,m,n}^{\ast}(0) & \text{otherwise}
            \end{cases}
        \end{equation*}
        where $\Phi_{f,m,n}^{\ast}$ is the dynatomic polynomial.
        Moreover, all of the roots of $G_d(m,n)$ as a polynomial in $c$ are simple.
    \end{thm}
    Additionally, Theorem \ref{thm1} leads to an explicit counting formula for Misiurewicz points (Corollary \ref{cor1}) and implies that Misiurewicz points are algebraic integers (Corollary \ref{cor2}).

    Thurston's rigidity theorem \cite{DH} says that, over the complex numbers, any fixed behavior of the critical points of a PCF map will be realized by only finitely many rational maps, up to equivalence and excepting Latt\`{e}s maps.
    Furthermore, in many cases, the equations defining these maps by critical orbit relations intersect transversely \cite{Epstein2,Epstein}.
    Theorem \ref{thm1}, gives a proof of Thurston's theorem for unicritical polynomials, including the transverality conclusion.
    In Section \ref{sect_thurston} we examine the failure of transversality in characteristic $p>0$. In particular, any prime which divides a discriminant of $G_d(m,n)$ for some $(m,n)$ is a characteristic where transversality does not hold. Silverman \cite{Silverman20} raised the question of describing this set of primes. We restrict to the case $m=0$ and give two reformulations of the problem in terms of periodic points of certain dynamical systems. These reformulations allows us to prove various properties about the primes where transversality fails and provide a connection to the dynamical Manin-Mumford problem. Additionally, we describe how the power dividing the discriminant is related to ramification in the number field generated by $G_d(0,n)$.

    In Section \ref{sect_bi}, we prove an explicit algebraic construction of all Misiurewicz points for the cubic bicritical family.
    \begin{defn}
        Define
        \begin{equation*}
            T(m,n,z) = \begin{cases}
                \frac{\Phi^{\ast}_{g,m,n}(z)}{\Phi^{\ast}_{g,0,n}(z)}  &  \text{if } m\neq0 \text{ and } n \mid (m-1)\\
                \Phi^{\ast}_{g,m,n}(z) & \text{otherwise}
            \end{cases}
        \end{equation*}
        Note that $T(m,n,z)$ is a polynomial in $a,v,z$.
    \end{defn}
    \begin{thm}\label{thm_degree3}
        Let $(m_1,n_1)$ and $(m_2,n_2)$ be pairs of non-negative integers. If $n_1 \neq n_2$, or if $n_1=n_2$ and $n_1$ does not divide at least one of $(m_1-1)$ and $(m_2-1)$,
	    then the points on the variety
        \begin{equation*}
            V(T(m_1,n_1,a),T(m_2,n_2,-a)) \subset \A^2_{a,v}
        \end{equation*}
        are exactly the parameters for which $g_{a,v}$ is PCF with $(m_1,n_1)$ and $(m_2,n_2)$ as the exact periods of the critical orbits for $a$ and $-a$, respectively.

        If $n_1=n_2$ and both $n_1 \mid m_1-1$ and $n_1 \mid m_2-1$, then we have to remove all the points $(0,v)$ where $v$ ranges over the parameters from Theorem \ref{thm1}, where $0$ has exact period $n_1$. In particular, the points contained in the difference
        \begin{equation*}
            V(T(m_1,n_1,a),T(m_2,n_2,-a)) - V(T(0,n_1,v)) \subset \A^2_{a,v}
        \end{equation*}
        are exactly the parameters for which $g_{a,v}$ is PCF with $(m_1,n_1)$ and $(m_2,n_2)$ as the exact periods of the critical orbits for $a$ and $-a$, respectively.
    \end{thm}

\begin{code}
Unicritical definition
\begin{python}
def Gleason(self,N):
    R=self.base_ring()
    m=N[0]
    n=N[1]
    F=R(self.dynatomic_polynomial([m,n]).subs({x:0,y:1}))
    if m!=0 and (m-1)%n==0:
        F=F/(R(self.dynatomic_polynomial([0,n]).subs({x:0,y:1})))^(self.degree()-1)
    return(self.codomain().base_ring()(F))
\end{python}

Cubic bicritical definition
\begin{python}
def Gleason2(self,N,v):
    R=self.base_ring()
    m=N[0]
    n=N[1]
    F=f.dynatomic_polynomial([m,n])(v,1)
    if m!=0 and (m-1)%n==0:
        F=F._maxima_().divide((f.dynatomic_polynomial([0,n])(v,1)))[0].sage()
    return(self.codomain().base_ring()(F))
\end{python}

\end{code}

\section{Background and Discussion}
    There is a growing collection of results on the set of post-critically finite polynomials in the moduli space of all polynomials \cite{Baker7, FG, Ingram2, Ingram5}. The earliest result relating to transversality is perhaps due to Gleason that
    \begin{equation*}
        f_{2,c}^n(0) = 0
    \end{equation*}
    has only simple roots as a polynomial in $c$, since when reducing modulo 2, $(f_{2,c}^n(0))'$ is always 1. Epstein \cite{Epstein} gives an algebraic proof of this fact for $f_{d,c}$ (and for any prime power polynomial).

    One approach to studying PCF maps is to fix a critical portrait and to describe all of the maps with that portrait.
    The \emph{critical portrait} of a map is the (weighted) directed graph whose vertices are the points in the orbits of the critical points (weighted by multiplicity) and whose directed edges are defined by the map, i.e., there is an arrow $P \to Q$ if and only if $Q$ is the image of $P$.
    \begin{exmp}
        The following is a critical portrait for a critical map with two critical points $P,Q$, one of which if fixed, one of which is preperiodic.
        \begin{equation*}
            \xygraph{
            !{<0cm,0cm>;<1cm,0cm>:<0cm,1cm>::}
            !{(0,0) }*+{P}="a"
            !{(1,0) }*+{\bullet}="b"
            !{(2,0) }*+{\bullet}="c"
            !{(3,0) }*+{\bullet}="d"
            !{(4,0) }*+{Q}="e"
            "a":"b"
            "b":"c"
            "c":@(rd,ru)"d"
            "d":@(rd,ru)"c"
            "e":@(rd,ru)"e"
            }
        \end{equation*}
    \end{exmp}

    Any given map has finitely many critical points, so the critical portrait of a PCF map imposes finitely many relations.
    Thurston's theorem tells us that only finitely many maps of fixed degree (up to equivalence) satisfy these relations \cite{DH}. An important consequence of Thurston's theorem is that, in many cases, the subvarieties defined by the critical point relations intersect transversely \cite{Epstein2,Epstein}.

    The quadratic family $f_{2,c}$ has received extensive study in both algebraic and complex dynamical contexts. The Mandelbrot set is the set of complex $c$ values where the orbit of the critical point remains bounded. The values where the critical orbit is purely periodic ($f^n(0) = 0$ for some $n$) are the centers of the hyperbolic components of the Mandelbrot set \cite{DHI,DHII,Milnor,Schleicher}.
    There are several methods to approximate the $c$ values which are centers of hyperbolic components:
    \begin{itemize}
        \item Apply Newton's method to the defining polynomial relation
        \item The Hubbard-Schleicher spider algorithm \cite{HS}
    \end{itemize}
    In particular, the Hubbard-Schleicher spider algorithm  
    allows one to compute the $c$ values corresponding to a particular combinatorics of the critical point orbit through successive approximation and can be generalized to both preperiodic orbits and the higher degree unicritical polynomials $f_{d,c}$.  Eberlein's thesis addresses the details for $f_{d,c}$ and the Multibrot set, the generalization of the Mandelbrot set \cite{E}.

    The (algebraic) study of Misurewicz points has relied on studying the polynomials $f^{m+n}(0) - f^n(0)$. These polynomials are critical orbit relations whose roots contain the points of exact period $(m,n)$. However, this is akin to studying the $n$th-roots of unity by examining $x^n-1$ instead of the $n$th cyclotomic polynomial. In this article we provide the equivalent of the $n$th cyclotomic polynomial for Misiurewicz points. This will facilitate the study of algebraic properties of the points with exact period $(m,n)$, i.e., algebraic properties of centers of hyperbolic components.

    While the following is certainly not a complete review of the results on Misiurewicz points, it does illustrate that while there are many results concerning Misiurewicz points, few of them are algebraic in nature.
    Pastor-Romera-Montoya count the real Misiurewicz points for $f_{2,c}$ \cite{PRM2} and Douady-Hubbard show that the complex Misiurewicz points are dense on the boundary of the Mandelbrot set and are branch tips, centers of spirals and points where branches meet \cite{DHI,DHII}. Eberlein's thesis shows for $f_{d,c}$ that the corresponding periodic cycle for Misiurewicz points are repelling \cite{E}. From the combinatorial description of Douday-Hubbard of the dynamics of post-critically finite polynomials by associating to each filled Julia set a Hubbard tree \cite{DHI} Poirier gives a complete classification of arbitrary post-critically finite polynomials \cite{Poirier}. Favre-Gauthier recently studied the distribution of Misiurewicz points with respect to the bifurcation measure and relies on transversality in a crucial way \cite{FG}.

    Even less is known in the bicritical cubic case.
    Silverman \cite{Silverman18} gives an algebraic proof of transversality for $g_{a,v}$ with both critical points periodic or with preperiod at most $1$.

    Fakkhruddin \cite{Fakkhruddin} recently used the simplicity of Misiurewicz points (transversality) to verify the dynamical Mordell-Lang and Manin-Mumford conjectures for generic endomorphisms of $\P^N$.

\section{Unicritical Polynomials}\label{sect_uni}

    In this section we construct a polynomial in $c$ whose roots are exactly the $c$ values for which the critical point $0$ has exact period $(m,n)$ for $f_{d,c}\left( z \right) =z^d+c$. For $m>0$, these are exactly the Misiurewicz points. We first need two lemmas concerning multiplicities of roots.
    Let
    \begin{equation*}
        F_k = \frac{f_{d,c}^{m+k}(0) - f_{d,c}^{m}(0)}{f_{d,c}^{m+k-1}(0) - f_{d,c}^{m-1}(0)}
    \end{equation*}
    which is a polynomial in $c$. We denote $a_k(c)$ as the multiplicity of $c$ as a root of $F_k$.

    \begin{lem}\label{lem1}
        For $k \mid n$
        \begin{equation*}
            a_k(c) = \begin{cases}
                        d-1 & f_{d,c}^{\ell}(0) = 0 \quad \text{and} \quad \ell \mid k,m-1\\
                        1 & \text{otherwise}.
                     \end{cases}
        \end{equation*}
    \end{lem}
    \begin{proof}
        Since $f_{d,c}(z) = z^d+c$ we can write
        \begin{equation*}
            f_{d,c}^{m+k}(0) - f_{d,c}^{m}(0) = (f_{d,c}^{m+k-1}(0))^d - (f_{d,c}^{m-1}(0))^d
        \end{equation*}
        so we have
        \begin{equation*}
            F_k = \prod_{0 < j < d} (f_{d,c}^{m+k-1}(0) - \zeta^jf_{d,c}^{m-1}(0))
        \end{equation*}
        where $\zeta$ is a primitive $d\tth$ root of unity.
        Any common zero of
        \begin{equation*}
            (f_{d,c}^{m+k-1}(0) - \zeta^i f_{d,c}^{m-1}(0)) \quad \text{and} \quad (f_{d,c}^{m+k-1}(0) - \zeta^j f_{d,c}^{m-1}(0))
        \end{equation*}
        for $i \neq j$ must be zeros of $f_{d,c}^{m+k-1}(0)$ and $f_{d,c}^{m-1}(0)$. In other words, $0$ is periodic of period dividing both $(m-1)$ and $(m+k-1)$, i.e. both $(m-1)$ and $k$. Such a common zero will, therefore, occur in each of the $(d-1)$ factors $(f_{d,c}^{m+k-1}(0) - \zeta^j f_{d,c}^{m-1}(0))$, and this is the only way that
        \begin{equation*}
            F_k = \frac{f_{d,c}^{m+k}(0) - f_{d,c}^{m}(0)}{f_{d,c}^{m+k-1}(0) - f_{d,c}^{m-1}(0)}
        \end{equation*}
        has a non-simple zero. From Epstein \cite{Epstein} the zeros of $(f_{d,c}^{m+k-1}(0) - \zeta^j f_{d,c}^{m-1}(0))$ are simple and the zeros of $f_{d,c}^{\ell}(0)$ are simple. Thus, the multiplicities of zeros of
        \begin{equation*}
            F_k = \prod_{\zeta^j \neq -1}(f_{d,c}^{m+k-1}(0) - \zeta^jf_{d,c}^{m-1}(0))
        \end{equation*}
        for $f_{d,c}^{\ell}(0)=0$ and $\ell \mid k,m-1$ are
        \begin{equation*}
            (d-1).
        \end{equation*}
        Otherwise, the zeros are simple zeros.
    \end{proof}
    \begin{lem} \label{lem2}
        If $k \mid n$ and $a_k(c) \neq 0$, then $a_k(c) =a_n(c)$.
    \end{lem}
    \begin{proof}
        We have two cases. If $a_k(c) = (d-1)$, then $f_{d,c}^{\ell}(0) = 0$ for some $\ell \mid k,m-1$. Thus, $\ell \mid n$ and $a_n(c) = d-1$.

        If $a_k(c)=1$, then $c$ is a zero of $f_{d,c}^{m+k}(0) - f_{d,c}^m(0)$ and $0$ has period $(m,k)$ for $f_{d,c}$. Since $k \mid n$, $c$ also has period $(m,n)$. Now assume that there is an $\ell$ such that $f_{d,c}^{\ell}(0)=0$ and $\ell \mid n,m-1$. Since $0$ is also period ${m,k}$ we also have $\ell \mid k$. But that is a contradiction to $a_k(c) = 1$, so we must have $a_n(c) = 1$.
    \end{proof}

    \begin{proof}[Proof of Theorem \ref{thm1}]
        Assume first that $m=0$. Then we have
        \begin{equation*}
            G_d(0,n) = \Phi_{f,0,n}^{\ast}(0) = \Phi^{\ast}_{f,n}(0) = \prod_{k \mid n} f_{d,c}^k(0)^{\mu(n/k)}.
        \end{equation*}
        We write the principal divisor associated to $G_d(0,n)$ as roots and multiplicities as
        \begin{equation*}
            \Gamma_f(0,n,c) = \sum_{k \mid n} \mu(n/k)a_k(c) (c). 
        \end{equation*}
        By Epstein, \cite{Epstein}, each $f_{d,c}^k(0)=0$ has only simple roots so that $a_k(c) = 1$ for all $k$ where $0$ has period $k$ for $f_{d,c}$ and $0$ otherwise. A standard property of the M\"obius function is that
        \begin{equation}\label{eq_mobius}
            \sum_{k \mid n} \mu(n/k) = \begin{cases}
              0 & n \neq 1\\
              1 & n=1
            \end{cases}.
        \end{equation}
        Thus, by (\ref{eq_mobius}) we have
        \begin{equation*}
            \sum_{k \mid n} \mu(n/k)a_k(c) = \begin{cases}
              0 & 0 \text{ does not have exact period } n \text{ for } f_{d,c}\\
              1 & 0 \text{ has exact period } n\text{ for } f_{d,c}.
            \end{cases}
        \end{equation*}

        Now assume that $m\neq 0$. We use the definition of preperiodic dynatomic polynomials from \cite{Hutz} as
        \begin{align*}
            \Phi^{\ast}_{f,m,n}(0)&= \prod_{k \mid n} \left(\frac{f_{d,c}^{m+k}(0) - f_{d,c}^{m}(0)}{f_{d,c}^{m+k-1}(0) - f_{d,c}^{m-1}(0)}\right)^{\mu(n/k)} = \prod_{k \mid n} F_k^{\mu(n/k)}.
        \end{align*}

        We notate the principal divisor of $F_k$ as $\Gamma_k$ as a zero cycle of roots and multiplicities.
        \begin{equation*}
            \Gamma_k := \sum_{F_k(c) = 0} a_k(c) (c)
        \end{equation*}
        where $a_k(c)$ is the multiplicity of $c$ in $F_k$.

        We can then write the zero cycle associated to $\Phi^{\ast}_{f,m,n}$ as
        \begin{equation*}
             \sum_{k \mid n} \mu(n/k)\Gamma_k.
        \end{equation*}
        We need to consider the following sum of multiplicities for every $c$ with $F_n(c) = 0$
        \begin{equation} \label{eq5}
            \sum_{k \mid n} \mu(n/k)a_k(c)
        \end{equation}
        and show that the only non-zero values are those for $0$ with exact period $(m,n)$ for $f_{d,c}$ and that multiplicity is $1$.

        By Lemmas \ref{lem1} and \ref{lem2} (and possible scaling $n$), (\ref{eq5}) is a sum of constants (either $1$ or $(d-1)$). By (\ref{eq_mobius})
        we have (\ref{eq5}) is 0 unless the sum has a single non-zero term. That occurs in the multiplicity $(d-1)$ case when $n$ is the exact period of $0$ for $f_{d,c}$ and in the simple root case when $0$ has exact period $(m,n)$ for $f_{d,c}$
        \begin{equation*}
            \sum_{k \mid n} \mu(n/k)a_k(c) =\begin{cases}
                            d-1 & 0 \text{ has exact period $n$ and } n \mid m-1\\
                            1 & 0\text{ has exact period } (m,n).
                         \end{cases}
        \end{equation*}
        In the multiplicity $(d-1)$ case, this $c$ value has $0$ with exact period $(0,n)$, so by the $m=0$ case we must divide by
        \begin{equation*}
            (\Phi^{\ast}_{f,0,n})^{d-1}.
        \end{equation*}
    \end{proof}
    From the formula in Theorem \ref{thm1}, it is straightforward to count the Misiurewicz points for $f_{d,c}(z)= z^d+c$. However, note that for $d>2$, the form $f_{d,c}$ provides a $(d-1)$-to-1 cover of the parameter space of unicritical polynomials
    \begin{cor}\label{cor1}
        The number of $(m,n)$ Misiurewicz points for $f_{d,c}$ is
        \begin{equation*}
            M_{m,n} = \begin{cases}
              \sum_{k \mid d} \mu(n/k)d^{k-1} & m=0\\
              (d^m - d^{m-1}-d+1)\sum_{k \mid d} \mu(n/k)d^{k-1} & m \neq 0 \text{ and } n \mid (m-1)\\
              (d^m - d^{m-1})\sum_{k \mid d} \mu(n/k)d^{k-1} & \text{otherwise}
            \end{cases}
        \end{equation*}
    \end{cor}
    \begin{proof}
        We take $\Phi_{f,m,n}(0)=f_{d,c}^{m+n}(0) - f_{d,c}^m(0)$ so that
        \begin{equation*}
            \deg(\Phi_{f,m,n}(0)) = d^{m+n-1}.
        \end{equation*}
        Notice the $-1$ in the exponent, which occurs since the first iterate only contains $c$, and the second iterate $c^d$, etc.

        If $m=0$, then
        \begin{align*}
             \deg(\Phi^{\ast}_{f,0,n}(0))
                &= \sum_{k \mid n} \mu(n/k) \deg(\Phi_{f,0,k})\\
                &= \sum_{k \mid n} \mu(n/k) d^{k-1}.
        \end{align*}
        If $m \neq 0$, then
        \begin{align*}
            \deg(\Phi^{\ast}_{f,m,n}(0))
                &= \sum_{k \mid n} \mu(n/k) \left(\deg(\Phi_{f,m,k}) - \deg(\Phi_{f,m-1,k})\right)\\
                &= \sum_{k \mid n} \mu(n/k) \left(d^{m+k-1} - d^{m+k-2}\right)\\
                &= (d^{m} - d^{m-1})\sum_{k \mid n} \mu(n/k) d^{k-1}.
        \end{align*}
        If in addition $n \mid (m-1)$, then we must subtract
        \begin{equation*}
            (d-1)\deg(\Phi_{f,0,n}^{\ast}) = (d-1)\sum_{k \mid n} \mu(n/k) d^{k-1}.
        \end{equation*}
    \end{proof}

    While it is an easy deduction from the defining equations, it is worth emphasizing that Misiurewicz points are algebraic integers.
    \begin{cor}\label{cor2}
        All Misiurewicz points are algebraic integers.
    \end{cor}
    \begin{proof}
        The polynomials $G_d(m,n)$ are monic.
    \end{proof}

    While it is believed that for $d=2$, $G_2(0,n)$ is irreducible for all $n$, it is not clear what happens for $m >0$ or for $d>2$. Figure \ref{figure1} shows that the irreducibility behavior appears quite complicated and it would interesting to understand for which $(m,n)$ and $d$ $G_d(m,n)$ is irreducible.

    \begin{figure}\label{figure1}
    \caption{Irreducibility of $G_d(m,n)$ For Pairs $(m,n)$}
    \begin{equation*}
        \begin{tabular}{|l|l|l|}
            \hline
            $d$ & Irreducible & Reducible\\
            \hline
            2 & $\{(m,n) \mid 0 \leq m \leq 4, 1 \leq n \leq 4\}$ & $\emptyset$\\
            \hline
            3 & $\{(m,n) \mid 0 \leq m \leq 3, 1 \leq n \leq 3\} \cup (4,2)$ & $\emptyset$\\
            \hline
            4 & $(0,3), (2,3), (3,2)$  & $(0,2), (2,1),(2,2),(3,1),(4,1)$\\
              &     & (2 factors)\\
            \hline
            5 & $(0,2),(0,3), (2,1), (3,1), (2,2)$ & $\emptyset$ \\
            \hline
            6 & (0,3) & (0,2),(2,1),(3,1),(2,2)\\
              &     & (3 factors)\\
            \hline
            7 & (2,1), (3,1), (2,2)  & (0,2), (0,3)\\
              &               & (3 factors)\\
            \hline
            8 & (0,3) & (0,2),(2,1)\\
            \hline
            9 & (0,2),(0,3) & (2,1)\\
            \hline
            10 & (0,3) & (0,2), (2,1)\\
            \hline
        \end{tabular}
    \end{equation*}
    \end{figure}

\begin{code}
\begin{python}
R.<c>=PolynomialRing(QQ,1)
d=5
P1.<x,y>=ProjectiveSpace(R,1)
H=Hom(P1,P1)
f=H([x^d + c*y^d,y^d])

Gleason(f,[0,2]).factor()
\end{python}
\end{code}

\section{Failure of Transversality over finite fields for unicritical polynomials} \label{sect_thurston}
    For the maps $f_{d,c}(z)$, Thurston's transversality statement says that $G_d(m,n)$ has only simple roots.  Since $G_d(m,n)$ is a polynomial in $c$, we can compute its discriminant. A polynomial discriminant is non-zero if and only if the polynomial has simple roots. Since the discriminant is an integer it will have prime divisors, unless of course it is $0, \pm 1$.
    \begin{defn}
        \begin{equation*}
            D_d(m,n) = \Disc(G_d(m,n)).
        \end{equation*}
    \end{defn}
    Thus, for each $n$ we will get some set of primes that divide $D_d(m,n)$ and, hence, primes where transversality fails over $\overline{F_p}$.  This problem is originally posed by Silverman \cite{Silverman20} to examine the primes dividing the discriminants of $f_{d,c}^n(0)$. The two problems are in fact equivalent. We provide the first steps towards a resolution by translating the problem into a problem about periodic points of dynamical systems in two different ways, one of which is related to the dynamical Manin-Mumford problem.

    Here are the first few $D_2(0,n)$
    \begin{equation*}
        \begin{tabular}{|l|l|}
        \hline
        n & $D_2(0,n)$\\
        \hline
        1 & 1 \\
        \hline
        2 & 1 \\
        \hline
        3 & $-1 \cdot 23$\\
        \hline
        4 & $23 \cdot 2551$\\
        \hline
        5 & $13 \cdot 24554691821639909$ \\
        \hline
        6 & $-1 \cdot 13^2 \cdot 949818439 \cdot 6488190752068386528993226361$ \\
        \hline
        7 & $-1 \cdot 8291 \cdot 9137 \cdot 420221 \cdot 189946395389 \cdot$\\
          & $4813162343551332730513 \cdot 2837919018511214750008829 \cdot$\\
          & $1858730157152877176856713108209153714699601$ \\
        \hline
        \end{tabular}
    \end{equation*}
    Computing the discriminant and factoring could be computed for a few more $n$, but certainly $n=15$ is already out of reach of reasonable computing power since the discriminant is growing with the degree of $D_d(m,n)$ which is growing doubly exponentially. Further data was generated using the reformulating as a dynamical system and the computer algebra system Sage \cite{sage}. For example, the following are all primes $p < 10000$ with $n < 100$ and $p \mid \mid D_2(0,n)$, i.e., $p \mid D_2(0,n)$ and $p^2 \nmid D_2(0,n)$.
    \begin{equation*}
        \begin{tabular}{|l|l||l|l|}
            \hline
            $n$ & $p$ & $n$ & $p$\\
            \hline
            3 & 23  & 24 & 2087, 3973\\
            \hline
            4 & 24, 2551 & 25 & 1289, 6449\\
            \hline
            5 & 13 & 27 & 2287\\
            \hline
            7 & 8291, 9137 & 31 & 3467, 2358\\
            \hline
            8 & 1433 & 32 & 1097\\
            \hline
            9 & 137 & 33 & 2153\\
            \hline
            11  & 547, 1613 & 35 & 1777\\
            \hline
            12 & 211 & 36 & 2953\\
            \hline
            13 & 317 & 37 & 5023\\
            \hline
            14 & 431, 2179 & 49 & 8693\\
            \hline
            19 & 251, 1249 & 63 & 5581\\
            \hline
        \end{tabular}
    \end{equation*}
    Notice that an entry for $n=6$ and $p=13$ is missing from the table. This is because over $\F_{13}$ transversality holds, but over $\F_{13^2}$ it does not, i.e., the double root is over an extension of $\F_{13}$, so we will have $p^2 \mid D_2(0,6)$. Such primes are not included in the table since it was infeasible to extend the brute-force search over field extensions for large primes. In Proposition \ref{prop2} we describe the power of $p$ dividing $D_n$.

    From this computational data, there are several things that seem to be true. It seems like there is a primitive prime divisor for every $n \geq 3$, that is a prime which appears at $n$ but not for any $i < n$, and that the density of primes which occurs is $0$. However, both of these statements seems quite difficult to prove due mainly to the possibility of the $c$ value occurring in some arbitrarily large algebraic extension. Silverman proposed an addition question:
    Fix a prime $p$ and consider the set
    \begin{equation*}
        S_p = \{ n \col \Disc(f_{d,c}^n(0)) \equiv 0 \pmod{p}\}.
    \end{equation*}
    Since $f_{d,c}^{kn}(0) \mid f_{d,c}^n(0)$ for all $k$ (since they divide in $\Z[c]$) is $S_p$ multiplicatively finitely generated? In other words, is there a finite set $\{n_1,\ldots,n_r\}$ such that $S_p = \cup_i n_i\N$. The difficulty with the question is that as $n$ increases so does the degree of the extension $\F_{p^k}$ where the double root would be defined, see Corollary \ref{cor:Conditional_bound}. While it seems possible that a double root could exist over a very large extension, the probability that it occurs decreases with $k$. We show that the power of $p$ dividing $D_n$ is a multiple of $d-1$, so that $d-1$ is the typical occurrence, and the particular field extensions where the double roots occur depends on how $x^n-1$ splits in $\F_p$. Surprisingly, $p=13$ for $n=6$ and $d=2$ is the only example we have, other than the special primes $p \mid d-1$ described in Corollary \ref{cor3}, for which the power is $t(d-1)$ for $t> 1$. Perhaps this is merely reflection of probability, or perhaps, there are only finitely many exceptions. If one could bound $t$ for which $p^{t(d-1)}$ occurs, then the answer to Silverman's question would be yes, each prime divides only finitely many $D_n$. However, it is unclear whether such a bound should exist. To illustrate that $p^{d-1}$ phenomenon, here are the first few $D_n$ for $d=3$
    \begin{equation*}
        \begin{tabular}{|l|l|}
        \hline
        n & $D_3(0,n)$\\
        \hline
        1 & 1 \\
        \hline
        2 & $-1 \cdot 2^2$ \\
        \hline
        3 & $2^8 \cdot 229^2$\\
        \hline
        4 & $2^{24} \cdot 41^2 \cdot 101^2 \cdot 1163^2 \cdot 2136287^2$\\
        \hline
        5 & $2^{80} \cdot 23^2 \cdot$\\
          & $432725499932442096897356758710872300429969741351525112750876791^2$\\
        \hline
        \end{tabular}
    \end{equation*}
    The powers of $2$ are explained by Corollary \ref{cor3} and the exponents of $2 = (d-1)$ are explained by Proposition \ref{prop3}.

\subsection{Power of $p$ dividing $D_d(m,n)$}
    We first recall a little algebraic number theory.
    \begin{thm}[\textup{\cite[\S III.3]{LangANT}}] \label{thm:disc}
        Let $K/\Q$ be an algebraic number field of degree $n$ with ring of integers $\O_K$ and discriminant $D_K$. Let $\omega \in \O_K$ with minimal polynomial $w(x)$. Then
        \begin{equation*}
            \Disc(w(x)) = [\O_K \col \Z[\omega]]^2 D_K.
        \end{equation*}
    \end{thm}
    \begin{defn}
        A prime $p$ which divides the index $i(\omega) = [\O_K \col \Z[\omega]]$ for all $\omega \in \O_K$ is called an \emph{inessential discriminant divisor}.
    \end{defn}
    An inessential discriminant divisor is an obstruction to $K$ being monogenic, i.e. having an integral power basis. This is important since in the monogenic case we have Dedekind's theorem for how rational primes split in number fields (see \cite{LangANT}, Chapter 1, Proposition 25). It is interesting to note that all computed examples of field extensions generated by irreducible factors of $G_d(m,n)$ were monogenic.

    \begin{prop} \label{prop2}
        Let $g(c)$ be an irreducible component of $G_d(m,n)$ with discriminant $D$. Let $K$ be the extension of $\Q$ by $g(c)$ with ring of integer $\O_K$. Assume that $p^k \mid D$ and $p \nmid i(c_0)=[\O_K\col \Z[c_0]]$, where $c_0$ is a root of $g(c)$.
        Then,
        \begin{equation*}
            k \geq \sum_i (e_i-1)f_i
        \end{equation*}
        where the $e_i$ are the ramification indices and $\deg(q_i) = f_i$ are the residue degrees for $K$.
    \end{prop}
    \begin{proof}
        Since $p \nmid i(c_0)$ we apply Dedekind's theorem to get the factorization of $p$ in $\O_K$ from
        \begin{equation*}
            g(c) \equiv w_1^{e_1}\cdots w_g^{e_g} \pmod{p}
        \end{equation*}
        for some irreducible polynomials $w_i$. So that $p$ ramifies in $\O_{K}$ if and only if it divides $D$. The $e_i$ are the ramification indices and $\deg(w_i) = f_i$ are the residue degrees.

        Let $k_0$ be such that $p^{k_0} \mid D$ but $p^{k_0+1} \nmid D$. We have \cite[III \S 6]{Serre2}
        \begin{equation*}
            k_0 \geq \sum_i (e_i-1)f_i
        \end{equation*}
        with equality if $\gcd(e_i,p)=1$ for all $i$.
    \end{proof}
    It is well known that the discriminant of a product has the following form
    \begin{equation} \label{eq6}
        \Disc(fg) = \Disc(f)\Disc(g)\Res(f,g)^2.
    \end{equation}
    Thus, we can use Proposition \ref{prop2} to examine the powers of $p$ dividing $G_n$ be examining each irreducible factor, see Example \ref{exmp1}.

    \begin{rem}
        While it is believed that $G_2(0,n)$ is irreducible, it is clear from Figure \ref{figure1} that $G_n$ is not always irreducible. Thus, it is possible that there can be contributions from the resultant terms, see Example \ref{exmp1}.
    \end{rem}

    \begin{exmp}
        $13^2 \mid D_2(0,6)$ since it has a multiple root over $\F_{13^2}$. In particular, there are $4$ primes that lie above $13$ for $n=6$. Three of them have $e=1$ and the fourth has $e = 2$ and $f=2$. Giving a total exponent of $2$. Or, we could see that
        \begin{align*}
            G_2(0,6) &\equiv (c + 9)(c^2 + 3c + 1)^2(c^4 + c^3 + 4c^2 + 12c + 3)\\
                &(c^{18} +10c^{17} + 3c^{16} + 8c^{15} + 2c^{14} + c^{13} + 10c^{12} + 3c^{11} + 2c^{10} +\\
                &10c^9 + 11c^8 + 9c^7 + 5c^6 + 11c^5 + 4c^4 + 6c^3 + 4c^2 + 12c+ 1) \pmod{13}.
        \end{align*}
        Notice the degree $2$ polynomial which occurs to the power $2$.
    \end{exmp}

    \begin{prop} \label{prop3}
        For each prime divisor $p^k \mid D_d(0,n)$, $p \nmid i(c_0)$ for all roots $c_0$ of $G_d(0,n)$, we have $(d-1) \mid k$. Furthermore, the number of multiple roots is the determined by the factorization of $x^{d-1}-1$ modulo $p$, i.e., the $(d-1)\sst$ roots of unity modulo $p$.
    \end{prop}
    \begin{proof}
        Replacing $c^{d-1}$ with $t$ we can write $f_{d,c}^n(0)/c$ as a polynomial $h(t)$. So any multiple root of $f_{d,c}^n(0)$ is a multiple root of $h(t)$ and we apply the methods of Proposition \ref{prop2} to $h(t)$. Finally, we take the resulting power and multiply by $(d-1)$, (undoing the $c^{d-1} \mapsto t$ substitution). Thus, the power of $p$ dividing $D_d(0,n)$ must be a multiple of $d-1$.

        Each root of $h(t)$ upon reverting to $c$ corresponds to some number of roots of $f_{d,c}^n(0)$. Those roots are determine by the factorization of $c^{d-1} -1$ modulo $p$.

        By Corollary \ref{cor:disc_mult} $\prod_{k \mid n} D_d(0,k) = \Disc(f_{d,c}^n(0)$, so we have the statement of the proposition.
    \end{proof}
    \begin{rem}
        In particular, this implies that $p$ always occurs to the power at least $d-1$ in the discriminant, so it either ramifies to high degree, or has a high residue field degree.
    \end{rem}
    \begin{exmp}
        Consider $f_{5,c}(x) = x^5+c$. Then we have
        \begin{equation*}
            D_5(0,3) = 2^{48} \cdot 101^4 \cdot 431^4.
        \end{equation*}
        So we look at the factorization of $G_5(0,3)$ mod $p$ for $p=101, 431$. We have for $p=101$
        \begin{equation*}
            (c + 6)^2(c + 41)^2 (c + 60)^2 (c + 95)^2 \cdot g(c)
        \end{equation*}
        where $g(c)$ has only simple roots. For $p=431$ we have
        \begin{equation*}
            (c^2 + 54c + 165)^2 (c^2 + 377c+ 165)^2 \cdot h(c)
        \end{equation*}
        where $h(c)$ has only simple roots. This phenomenon is because if we replace $c$ by $c^{d-1}$ in $G_d(0,n)$, then that polynomial has a single double root. Then the final splitting depends on the roots of unit in $\F_p$. For $p=101$ we have four $4\tth$ roots of unity
        \begin{equation*}
            (c^4-1) \equiv (c + 1) (c + 10) (c + 91) (c + 100) \pmod{101}
        \end{equation*}
        and for $p=431$ has only two $4\tth$ roots of unity
        \begin{equation*}
            (c^4-1) \equiv (c + 1) (c + 430) (c^2 + 1) \pmod{431}.
        \end{equation*}
    \end{exmp}

    \begin{exmp} \label{exmp1}
        Consider $f_{7,c} = x^7 + c$. Then $G_7(0,3)$ is reducible over $\Q$, so we expect exponents larger than $d-1$.
        In particular, we have
        \begin{equation*}
            G_7(0,3) = (c^6 - c^3 + 1)(c^6 + c^3 + 1)(c^{36} + 6c^{30} + 14c^{24} + 15c^{18} + 6c^{12} + 1)
        \end{equation*}
        and
        \begin{equation*}
            D_7(0,3) = 2^{48} \cdot 3^{54} \cdot 19^{12} \cdot 14731^6.
        \end{equation*}
        Using Corollary \ref{cor:disc_mult} and equation (\ref{eq6}) we can explain the powers that occur in $D_7(0,3)$.
        The discriminants of the three irreducible factors of $G_7(0,3)$ respectively are
        \begin{equation*}
            -3^9, \qquad -3^9, \qquad 2^{36} \cdot 3^{36} \cdot 14731^6.
        \end{equation*}
        The three pairs of resultants are
        \begin{equation*}
            2^6, \qquad 19^3, \qquad 19^3.
        \end{equation*}
        The fact that all exponents in $D_7(0,3)$ are multiples of $6$ is from Proposition \ref{prop3} and the high powers of $2,3$ are from Corollary \ref{cor3} with the extra powers of 3 occurring from the inequality of Proposition \ref{prop3} since in that case $\gcd(e_i,p) \neq 1$.

    \end{exmp}
    This example emphasizes yet again, that the exponent $2$ for $p=13$ for $d=2, n=6$ is special.

\subsection{Reformulation as a dynamical system I} \label{sect:reform1}
    We consider the following family of $2$ dimensional dynamical systems.
    \begin{defn}
        \begin{align*}
            F_{d,c}:\A^2 &\to \A^2\\
            (x,y) &\mapsto [x^d + c, dx^{d-1}y +1]
        \end{align*}
    \end{defn}
    We show that $p$ dividing $D_d(0,n)$ is equivalent to $(0,0)$ having minimal period $n$. This allows for several immediate consequences and justifies the computational data. First we clarify the implications of $p$ dividing the discriminant $D_d(0,n)$.
    \begin{lem} \label{lem4}
        $p \mid D_d(0,n)$ if and only if there exists a $c \in \overline{\F_p}$ such that $0$ is periodic of minimal period $n$ for $f_{d,c}$ and $(f_{d,c}^n(0))'=0$ in $\overline{\F_p}$.
    \end{lem}
    \begin{proof}
        $D_d(0,n)$ is the discriminant of $G_d(0,n)$. In particular, if $D_d(0,n) = 0$, there is some $c$ value which is a root of both $G_d(0,n)$ and its derivative. The roots of $G_d(0,n)$ are exactly the $c$ values for which $0$ has minimal period $n$ for $f_{d,c}$. Furthermore
        \begin{equation*}
            f_{d,c}^n(0) = \prod_{d \mid n} G_d(0,n).
        \end{equation*}
        Since $G_d(0,n)$ and $G_d'(0,n)$ are both $0$ at $c$, then the weaker statement $(f_{d,c}^n(0))'=0$ is also true.
    \end{proof}

    \begin{thm} \label{thm_reform}
        $p \mid D_d(0,n)$ if and only if $(0,0)$ is periodic of minimal period $n$ for $F_{d,c}(x,y)$ for some $c \in \overline{\F_{p}}$.
    \end{thm}
    \begin{proof}
        Assume that $(0,0)$ is periodic of minimal period $n$ for some $c$ in $\overline{\F_p}$. Then there is some $k$ so that $c \in \F_{p^k}$ and that $f_{d,c}^n(0) \equiv 0$ so that $0$ is periodic of period dividing $n$ for $f_{d,c}$. Assume that the period of $0$ is strictly smaller than $n$, then we have the orbit
        \begin{equation*}
            (0,0) \mapsto (c,1) \mapsto (0,\alpha) \mapsto (c,1) \mapsto \cdots
        \end{equation*}
        so that $\alpha = 0$ and $0$ must be minimal period $n$. So the question is then of the derivatives. We have
        \begin{equation*}
            f_{d,n}(0) = f_{d,n-1}^d(0) +c
        \end{equation*}
        so that
        \begin{equation*}
            \frac{d f_{d,c}^n(0)}{dc} = d(f_{d,c}^{n-1}(0))^{d-1}(f_{d,c}^{n-1}(0))' + 1.
        \end{equation*}
        Labelling $f_{d,n-1} = x$ and its derivative $y$, we see that
        \begin{equation*}
            F_{d,c}^n (0,0) = (f_{d,c}^n(0), (f_{d,c}^n(0))').
        \end{equation*}
        Thus, if the minimal period of $(0,0)$ is $n$ in $\F_{p^k}$, then by Lemma \ref{lem4} $n$ is the smallest integer such that both $f_{d,c}^n(0)$ and $(f_{d,c}^n(0))'$ are $0$ mod $p$ and $p \mid D_d(0,n)$

        If $p \mid D_d(0,n)$, then by Lemma \ref{lem4} there is some $k$ and some $c \in \F_{p^k}$ such that $f_{d,c}^n(0) \equiv (f_{d,c}^n)'(0) \equiv 0$ and in particular, $0$ is period $n$ for $f_{d,c}$. Thus $(0,0)$ is at least minimal period $n$ for $F_{d,c}(x,y)$. We again consider the orbit of $(0,0)$
        \begin{equation*}
            (0,0) \mapsto (c,1) \mapsto (0,\alpha) \mapsto (c,1) \mapsto \cdots
        \end{equation*}
        for some $\alpha$. Since $p \mid D_n$ and
        \begin{equation*}
            F_{d,c}^n (0,0) = (f_{d,c}^n(0), (f_{d,c}^n(0))'),
        \end{equation*}
        we must have $\alpha =0$, so that $(0,0)$ is periodic with minimal period at most $n$. Thus, $(0,0)$ is periodic with minimal period $n$.
    \end{proof}

    Similar to Gleason's original proof of transversality we can show that primes dividing $d$, never divide the discriminant,
    \begin{cor}
        For all $p \mid d$, $p \nmid D_d(0,n)$ for all $n$.
    \end{cor}
    \begin{proof}
        If there is $c$ value for which $(0,0)$ is periodic of minimal period $n$ over $\overline{\F_p}$, then it must also be periodic modulo $p$. But
        \begin{equation*}
            F_{d,c}(x,y) \equiv (x^d +c, 1) \pmod{p}
        \end{equation*}
        so that $y \neq 0$ for all $n$. Thus, $(0,0)$ is never periodic.
    \end{proof}

    \begin{cor} \label{cor3}
        If $p\mid d-1$, then $p^k \mid D_d(0,n)$ for all $n \geq 2$ where
        \begin{equation*}
            k \geq \sum_{t \mid n} \mu(n/t)d^{t-1}.
        \end{equation*}
    \end{cor}
    \begin{proof}
        The powers of $c$ in $f_{d,c}^n(0)$ are all congruent to $1 \pmod{d-1}$, so that after factoring out $c$, they are all divisible by $d-1$. Since $p$ divides $d-1$ we see that
        \begin{equation*}
            \frac{f_{d,c}^n(0)}{c} \equiv h(c)^{d-1} \pmod{p}
        \end{equation*}
        for some polynomial $h$. Thus, every root is a multiple root and, since $\gcd(d-1,p) > 1$ we have a strict inequality in Proposition \ref{prop2} \cite[III \S 6]{Serre2}, so that
        \begin{equation*}
            k \geq \sum_i e_if_i = \deg(G_n).
        \end{equation*}
        Thus, the power of $p$ dividing $D_d(0,n)$ is at least the degree of $G_d(0,n)$ which by the $m=0$ case of Corollary \ref{cor1} is
        \begin{equation*}
            \sum_{k \mid n} \mu(n/k)d^{k-1}.
        \end{equation*}
    \end{proof}

    \begin{cor}\label{cor:disc_mult}
        $\Res(G_d(0,n), G_d(0,m))=\pm 1$ for all $n \neq m	$. Moreover,
        $\Disc(f_{d,c}^n(0)) = \prod_{k \mid n} D_d(0,k)$.
    \end{cor}
    \begin{proof}
        The statement
        \begin{equation*}
            \Res(G_d(0,n),G_d(0,m)) = \pm 1
        \end{equation*}
        for $n \neq m$ is the same as saying there is no prime such that $G_d(0,n)$ and $G_d(0,m)$ have a common root. By Theorem \ref{thm_reform} this would mean that there is some $c$ value over $\overline{\F_p}$ for which $(0,0)$ is minimal period $n$ and $m$, which cannot happen if $n \neq m$.

        For the second part, we know that
        \begin{equation*}
            f_{d,c}^n(0) = \prod_{k \mid n} G_d(0,k)
        \end{equation*}
        and that
        \begin{equation*}
            \Disc(fg) = \Disc(f)\Disc(g)\Res(f,g)^2.
        \end{equation*}
    \end{proof}

    \begin{def}\label{def:DS}
      A sequence, $\left\lbrace a_n \right\rbrace$, of integers is called a \emph{divisibility sequence} if whenever $m \vert n$ then $a_m \vert a_n$.
    \end{def}

    \begin{cor}\label{cor:DS}
     Let $A_n: = \Disc\left(f^n_c \left( 0  \right)\right)$. The sequence $\left\lbrace A_n \right\rbrace$ is a divisibility sequence.
    \end{cor}

    \begin{proof}
     From Corollary \ref{cor:disc_mult} we know that $A_m = \prod_{k \vert m } D_d(0,k)$. So if $k\vert m $ then $D_d(0,k) \vert A_m$. Also, if $k \vert m $ then $k \vert n$, thus $A_m = \prod_{k \vert m } D_d(0,k) \vert A_n = \prod_{k \vert n } D_d(0,k)$.
    \end{proof}

    We finish this section with a remarks on Silverman's multiplicatively finite generation question.
    \begin{cor}\label{cor:Conditional_bound}
      Fix a prime $p$. If the powers of $p$ appearing in the sequence $D_d(0,n)$ are bounded, then the set
      \begin{equation*}
        S_p:= \left\{n: \Disc(f_c^n(0)): = A_n \equiv 0 \mod p\right\}
      \end{equation*}
      is multiplicatively finitely generated.
    \end{cor}
    \begin{proof}
        If the powers of $p$ that divide $D_d(0,n)$ are bounded, then there is some $k$ such that every multiple root $c$ is in $\F_{p^k}$. Since this is a finite set, the possible (minimal) period of $(0,0)$ is bounded, so that $p$ divides only finitely many $D_d(0,n)$, call then $D_d(0,n_1),\ldots, D_d(0,n_r)$. Then $S_p$ is generated by $\{n_1,\ldots, n_r\}$.
    \end{proof}

\subsection{Reformulation as a dynamical system II - Manin-Mumford}
    The reformulation of Section \ref{sect:reform1} restated the problem in terms of an infinite family of dynamical systems. In this section, we instead reformulate the problem in terms of a single $3$ dimensional dynamical system. However, instead of having to consider only the orbit of the point $(0,0)$, we now have to consider the orbits of the points $(0,c,0)$ for any $c$ value. So we trade an infinite family of functions on a single point, to a single function on a subvariety of points. Using this system we could prove many of the same corollaries as in Section \ref{sect:reform1} in much the same way. Instead, we content ourselves with showing the connection to the dynamical Manin-Mumford conjecture.

    The Manin-Mumford conjecture, proved by Raynaud \cite{Raynaud2, Raynaud}, states that a subvariety of an abelian variety contains a dense set of torsion points if and only if it is a torsion translate. Zhang  stated a dynamical version for polarized dynamical systems: A subvariety is preperiodic if and only if it contains a Zariski dense set of preperiodic points. Unfortunately, the conjecture is not true in this form \cite{ghioca3}. The conjecture is reformulated for endomorphisms and there are a few special cases known \cite{ghioca3}, but the problem over $\overline{\F_p}$ and for rational maps remains open.

    \begin{thm} \label{thm_reform2}
        A prime $p$ divides $D_d(0,n)$ if and only if the map
        \begin{equation*}
            R_d(x,y,z) = [y,y^d+y-x^d,dx^{d-1}z+1]
        \end{equation*}
        has a periodic point in $\overline{\F_p}$ of the form $(0,c,0)$ with minimal period $n$.
    \end{thm}
    \begin{proof}
        The orbit of $(0,c,z_0)$ for any $z_0$ under $F(x,y,z)$ is
        \begin{equation*}
            R_d^n(0,c,z_0) = (G_d(0,n),G_d(0,n+1),G_d'(0,n)).
        \end{equation*}

        By Theorem \ref{thm_reform}, $p \mid D_d(0,n)$ if and only if there is a $c \in \overline{\F_p}$ such that $0$ is minimal period $n$ for $F_{d,c}$.
    \end{proof}
    The map $R_d$ appears to be an example where the dynamical Manin-Mumford conjecture does not hold, at least for periodic points and subvarieties.
    \begin{cor}
        Fix a prime $p$. If $\{n \col p \mid D_d(0,n)\}$ is infinite, then $R_d$ has a non-periodic subvariety with infinitely many periodic points over $\overline{\F_p}$.
    \end{cor}
    \begin{proof}
        If $\{n \col p \mid D_d(0,n)\}$ is infinite, then there are infinitely many periodic points on the line $\{x=z=0\}$. However, that line is not preperiodic since
        \begin{equation*}
            (0,0,0) \to (0,0,1) \to (0,0,1) \to \cdots.
        \end{equation*}
    \end{proof}

\section{Bicritical polynomials}\label{sect_bi}
    \quad Now we turn to the case of cubic polynomials with $2$ free critical points. We use the following monic centered form:
    \begin{equation*}
        g_{a,v}(z) = z^3 - 3a^2z + (2a^3+v),
    \end{equation*}
    which has critical points $\pm a$ and marked critical value $g_{a,v}(a) = v$.
    Note that if $a = -a = 0$, then this form is $z^3 + v$, which we discussed above.
    The goal is to determine the pairs $(a,v)$ for which $a$ and $-a$ have finite (forward) orbit, the cubic Misiurewicz points.
    In Theorem \ref{thm_degree3} we construct the PCF cubic polynomials as the set where both $a$ and $-a$ have finite orbits.

    \begin{lem} \label{lem3}
        \begin{equation*}
            \frac{g_{a,v}^{m+k}(z) - g_{a,v}^{m}(z)}{g_{a,v}^{m+k-1}(z) - g_{a,v}^{m-1}(z)}
                =g_{a,v}^{m+k-1}(z)^2 + g_{a,v}^{m+k-1}(z)g_{a,v}^{m-1}(z) + g_{a,v}^{m-1}(z)^2 - 3a^2
        \end{equation*}
    \end{lem}
    \begin{proof}
        For notational convenience, write
        \begin{equation*}
            g_{a,v}^{m+k-1}(z) = F \quad \text{and} \quad g_{a,v}^{m-1}(z) = G.
        \end{equation*}
        Then we have
        \begin{align*}
            \frac{g_{a,v}^{m+k}(z) - g_{a,v}^{m}(z)}{g_{a,v}^{m+k-1}(z) - g_{a,v}^{m-1}(z)}
            &= \frac{F^3 - 3a^2F - (G^3 - 3a^2G)}{F-G}\\
            &= \frac{(F^3- G^3) - 3a^2(F-G)}{F-G}\\
            &= F^2 + FG + G^2 - 3a^2.
        \end{align*}
    \end{proof}
    Recall exact divisibility: we say $a^k \mid\mid b$ if $a^k \mid b$, but $a^{k+1} \nmid b$.
    \begin{prop} \label{prop1}
        The only zeros of $\Phi^{\ast}_{g,m,n}(a)$ that do not have $a$ with exact period $(m,n)$ for $g_{a,v}$ have $m\geq 1$ and are in one of the following cases:
        \begin{enumerate}
            \item $a=0$ and $n \mid m-1$.
            \item $(m,n)$ has $m \neq 0$, $n \mid m-1$, and $a$ has exact period $(0,n)$. Furthermore, $\Phi^{\ast}_{g,0,n}(a) \mid\mid \Phi^{\ast}_{g,m,n}(a)$.
            \item $(m,n)$ has $m \neq 0$, $g_{a,v}^{m-1}(a) = -a$, and $-a$ does not have exact period $(1,n)$.
        \end{enumerate}
    \end{prop}
    \begin{proof}
        If $m=0$, then by Silverman \cite{Silverman18} all the multiplicities are $1$ and multiplicity $1$ roots of dynatomic polynomials with $m=0$ have exact period $(0,n)$. So we may assume that $m \geq 1$.

        Assume first that $a=0$. We are working with the map $z^3 + v$ and we know when it is post-critically finite from Theorem \ref{thm1}. In particular, the only situation where the exact period is incorrect is for $n \mid m-1$ since
        \begin{equation*}
            \Phi^{\ast}_{g,0,n}(0)^2 \mid\mid \Phi^{\ast}_{g,m,n}(0),
        \end{equation*}
        but the denominator of $T(m,n,a)$ only contains $\Phi^{\ast}_{g,0,n}(0)$.

        We now assume that $a \neq 0$. By Lemma \ref{lem3}
        \begin{align}
            F_k(z) &:= \frac{g_{a,v}^{m+k}(z) - g_{a,v}^{m}(z)}{g_{a,v}^{m+k-1}(z) - g_{a,v}^{m-1}(z)} \label{eq3} \\
                &=g_{a,v}^{m+k-1}(z)^2 + g_{a,v}^{m+k-1}(z)g_{a,v}^{m-1}(z) + g_{a,v}^{m-1}(z)^2 - 3a^2. \label{eq1}
        \end{align}
        The zeros $(a,v)$ of $F_k(a)$ which do not have $a$ with exact period $(m,k)$ are the zeros of the denominator that are higher multiplicity zeros of the numerator, i.e. zeros of (\ref{eq1}) that are also zeros of the denominator of (\ref{eq3}).

        Consider $F_k(z)$ symbolically with $F=g_{a,v}^{m+k-1}(z)$ and $G= g_{a,v}^{m-1}(z)$. Then (\ref{eq1}) is
        \begin{equation} \label{eq4}
            F^2 + FG + G^2 - 3a^2.
        \end{equation}
        Since we are looking for zeros of (\ref{eq4}) which are also zeros of the denominator of (\ref{eq3}), we consider $F=G$. Then (\ref{eq4}) is zero when
        \begin{equation*}
            F^2 = G^2 = a^2.
        \end{equation*}
        Using $z=a$, we have two possibilities:
        \begin{equation*}
            g_{a,v}^{m-1}(a)= g_{a,v}^{m+k-1}(a)=a \quad \text{or} \quad g_{a,v}^{m-1}(a) = g_{a,v}^{m+k-1}(a)=-a.
        \end{equation*}
        Notice that
        \begin{equation*}
            g_{a,v}(z) - g_{a,v}(a) = (z-a)^2(z+2a).
        \end{equation*}
        Working modulo $(g_{a,v}^{m-1}(a)-a)$, we also have
        \begin{equation} \label{eq2}
            g_{a,v}(z) - g_{a,v}^m(a) \equiv (z-a)^2(z+2a) \pmod{(g_{a,v}^{m-1}(a)-a)}.
        \end{equation}

        Assume that $g_{a,v}^{m-1}(a)=g_{a,v}^{m+k-1}(a)= a$ so that there is a $t$ such that $t \mid k$, $t \mid (m-1)$, and $g_{a,v}^{t}(a) =a$. If we consider
        \begin{equation*}
            \tilde{F_k}(z) := \frac{g_{a,v}^{m+k}(z) - g_{a,v}^{m}(a)}{g_{a,v}^{m+k-1}(z) - g_{a,v}^{m-1}(a)}
        \end{equation*}
        then we can apply (\ref{eq2}) with $z=g_{a,v}^{m+k-1}$ to have
        \begin{align*}
            \tilde{F_k}(z) &\equiv \frac{(g_{a,v}^{m+k-1}(z) - a)^2(g_{a,v}^{m+k-1}(a) +2a)}{g_{a,v}^{m+k-1}(z) - a} \pmod{(g_{a,v}^{m-1}(a)-a)} \\
            &\equiv(g_{a,v}^{m+k-1}(z) - a)(g_{a,v}^{m+k-1}(z) +2a) \pmod{(g_{a,v}^{m-1}(a)-a)}.
        \end{align*}
        By Epstein \cite[Corollary 2]{Epstein} for any $\ell$, if $(a,v)$ is a point on $g_{a,v}^{\ell}(z)-a$, then it is a simple point, i.e., the curve is smooth at that point. Since $a \neq 0$, then $(g_{a,v}^{m+k-1}(a) +2a) \neq 0$. So we have each zero $(a,v)$ of $g_{a,v}^t(a)=a$ occurs with multiplicity $1$ in $\tilde{F_k}$ and, thus, multiplicity $1$ in $F_k$.
        Recall that
        \begin{equation*}
            \Phi^{\ast}_{g,m,n}(a) = \prod_{k \mid n} F_k^{\mu(n/k)}.
        \end{equation*}
        Thus, by (\ref{eq_mobius}), to have a zero of $\Phi^{\ast}_{g,m,n}(a)$ with $g_{a,v}^k(a)=a$ we must have the exact period of $a$ is $n$, and $n \mid (m-1)$. Thus, we see that the pairs $(a,v)$ where $a$ is strictly periodic are also zeros of $\Phi^{\ast}_{g,m,n}(a)$ of multiplicity $1$ and from Silverman \cite{Silverman18} they are multiplicity $1$, in $\Phi^{\ast}_{g,0,n}(a)$. Therefore, we have the exact divisibility as required.

        If $g_{a,v}^{m-1}(a)=g_{a,v}^{m+k-1}(a)= -a$, then
        \begin{equation*}
           g_{a,v}^{m+k-1}(a) = g_{a,v}^k(-a) = -a.
        \end{equation*}
        We can apply the same argument as the previous case to see that $-a$ must have exact period $n$ with $n \mid (m-1)$. Our assumption has the orbit of $a$ intersecting the orbit of $-a$; in particular, $a$ is in the strictly preperiodic portion of the orbit of $-a$. Then, $a$ has exact period $(m,n)$ if and only if $-a$ has exact period $(1,n)$.

    \end{proof}

    \begin{proof}[Proof of Theorem \ref{thm_degree3}]
        From the construction of $T(m,n,z)$ with dynatomic polynomials, it is clear that all the pairs $(a,v)$ with the appropriate critical point orbits are on the variety. What we need to see is that no additional points, points with ``smaller'' exact period (except possibly $a=0$), are on the variety.

        By Proposition \ref{prop1} we have four cases to consider:
        \begin{enumerate}
            \item $a=0$, $n_1=n_2$ and both $n_1 \mid m_1-1$ and $n_1 \mid m_2-1$.
            \item $(m_1,n_1)$ have $m_1 \neq 0$ and $n_1 \mid m_1-1$ and $a$ has exact period $(0,n_1)$.
            \item $(m_2,n_2)$ have $m_2 \neq 0$ and $n_2 \mid m_2-1$ and $-a$ has exact period $(0,n_2)$.
            \item $(m_1,n_1)$ have $m_1 \neq 0$ and $g_{a,v}^{m_1-1}(a) = -a$, and $-a$ does not have exact period $(1,n_1)$.
        \end{enumerate}
        We address each case separately.
        \begin{enumerate}
            \item In this case, there is a zero of $T(m_1,n_1,0)$ and $T(m_2,n_2,0)$ that has $a$ with exact period $(0,n_1)$. If $a=0$, then this is the map $g_{a,v}(z) = z^3 + v$, and we can compute exactly which such maps have $0$ with given period from Theorem \ref{thm1}. These are exactly the points described in the statement of the theorem.

            \item By the exact divisibility $\Phi^{\ast}_{g,0,n}(a) \mid\mid \Phi^{\ast}_{g,m,n}(a)$ from Proposition \ref{prop1}, the construction of $T(m,n,z)$ excludes all such points from the variety.

            \item Same as the previous case using $-a$ in place of $a$.

            \item If $-a$ does not have the correct period $(1,n_1) = (m_2,n_2)$, then the point $(a,v)$ is not on the variety $V(T(m_1,n_1,a),T(m_2,n_2,-a))$ since $T(m,n,z)$ is a dynatomic polynomial and its root will have period $(m,n)$. If $-a$ does have the correct period, then the point $(a,v)$ is on the variety.
        \end{enumerate}
    \end{proof}

    Silverman \cite{Silverman18} proves the transversality statement of Thurston's Theorem for $g_{a,v}(z)$ when $m_1,m_2\leq 1$. We do not address the extension of this problem here other than to mention that one needs to be careful about choosing the ``correct'' equations for a Thurston rigidity result. For example, when the two critical points, $\pm a$, both belong to the same cycle, we need a different set of equations than when they are part of separate cycles.
    \begin{exmp}
        Consider
        \begin{equation*}
            (m_1,n_1) = (2,2) \quad \text{and} \quad (m_2,n_2) = (1,2).
        \end{equation*}
        Consider
        \begin{equation*}
            a = \frac{i}{\sqrt[3]{4}} \quad \text{and} \quad v = -\frac{i}{\sqrt[3]{4}}.
        \end{equation*}
        Then we have the portrait
        \begin{equation*}
            \xygraph{
            !{<0cm,0cm>;<1cm,0cm>:<0cm,1cm>::}
            !{(0,0) }*+{a}="a"
            !{(1,0) }*+{-a}="b"
            !{(2,0) }*+{\bullet}="c"
            !{(3,0) }*+{\bullet}="d"
            "a":"b"
            "b":"c"
            "c":@(rd,ru)"d"
            "d":@(rd,ru)"c"
            }
        \end{equation*}
        We compute the jacobians to see that $(a,v)$ has multiplicity $>1$ for the variety
        \begin{equation*}
             V(T(2,2,a),T(1,2,-a))
        \end{equation*}
        but multiplicity 1 for the variety
        \begin{equation*}
             V(g_{a,v}(a)+a,T(1,2,-a)).
        \end{equation*}
\begin{code}
\begin{python}
R.<a,v>=PolynomialRing(QQ,2)
P1.<x,y>=ProjectiveSpace(R,1)
H=Hom(P1,P1)
f=H([x^3-3*a^2*x*y^2 + (2*a^3+v)*y^3,y^3])
N=[[2,2],[1,2]]
I0=R.ideal(Gleason2(f,N[0],a))
I1=R.ideal(f.nth_iterate_map(1).dehomogenize(1)[0](a)+a)
I2=R.ideal(Gleason2(f,N[1],-a))

J=I0+I1
print J.vector_space_dimension()
K=J.variety(ring=ComplexField())
print len(K)

J2=I0+I2
print J2.vector_space_dimension()
K2=J2.variety(ring=ComplexField())
print len(K2)

jac=jacobian(J.gens(),R.gens())
for d in K:
    aa=d[a]
    vv=d[v]
    print "---------------"
    print aa,vv
    print jac.determinant().subs({a:aa, v: vv})

jac=jacobian(J2.gens(),R.gens())
for d in K:
    aa=d[a]
    vv=d[v]
    print "---------------"
    print aa,vv
    print jac.determinant().subs({a:aa, v: vv})
\end{python}
or
\begin{python}
S.<x>=PolynomialRing(QQ)
K.<i>=NumberField(x^2+1)
S.<x>=PolynomialRing(QQ)
L.<w>=K.extension(x^3-4)
P1.<x,y>=ProjectiveSpace(L,1)
H=Hom(P1,P1)
aa=1/w*i
vv=-1/w*i
f=H([x^3-3*aa^2*x*y^2 + (2*aa^3+vv)*y^3,y^3])

f.orbit(P1(aa),4)

f.orbit(P1(-aa),3)
\end{python}
Some useful code to check orbit lengths and multiplicity.
\begin{python}
R.<a,v>=PolynomialRing(QQ,2)
P1.<x,y>=ProjectiveSpace(R,1)
H=Hom(P1,P1)
f=H([x^3-3*a^2*x*y^2 + (2*a^3+v)*y^3,y^3])

N=[[2,2],[1,2]]
I0=R.ideal(Gleason2(f,N[0],a))
I2=R.ideal(Gleason2(f,N[1],-a))

J=I0+I2
print J.vector_space_dimension()
K=J.variety(ring=ComplexField(100))
print len(K)

P1C.<x,y>=ProjectiveSpace(CC,1)
H=Hom(P1C,P1C)
jac=jacobian(J.gens(),R.gens())
for d in K:
    aa=d[a]
    vv=d[v]
    print "---------------___"
    print aa,vv
    print "jac:",jac.determinant().subs({a:aa, v: vv})
    fC=H([x^3-3*aa^2*x*y^2 + (2*aa^3+vv)*y^3,y^3])
    for m in range(0,3):
        for n in range(1,3):
            fC.nth_iterate(P1C(aa,1),m+n).dehomogenize(1)[0]).abs()
            if (fC.nth_iterate(P1C(aa,1),m).dehomogenize(1)[0] - fC.nth_iterate(P1C(aa,1),m+n).dehomogenize(1)[0]).abs() < 10^(-3):
                print "aa:",[m,n]
    for m in range(0,3):
        for n in range(1,3):
            if (fC.nth_iterate(P1C(-aa,1),m).dehomogenize(1)[0] - fC.nth_iterate(P1C(-aa,1),m+n).dehomogenize(1)[0]).abs() < 10^(-3):
                print "-aa:",[m,n]
\end{python}
\end{code}
    \end{exmp}

\bibliography{misiurewicz}
\bibliographystyle{plain}

\end{document}